\documentclass[11pt]{article}
\usepackage{algorithm,algorithmic}
\usepackage{amsmath,amsthm}
\usepackage{latexsym}
\usepackage{amssymb}
\usepackage{xspace}
\usepackage[hypertex]{hyperref}

\newif\ifnormalthm 
\normalthmfalse
\normalthmtrue 

\newtheorem{theorem}{Theorem}[section]
\newtheorem{Thm}[theorem]{Theorem}
\newtheorem{Lem}[theorem]{Lemma}

\theoremstyle{definition}

\newcommand{\K}{\ensuremath{\mathcal K}}

\def\deltil{\tilde{\delta}}

\newcommand{\ignore}[1]{}
\newcommand\E{\mbox{\bf E}}

\def\bx{\mathbf{x}}
\def\g{\mathbf{g}}
\def\x{\mathbf{x}}
\def\y{\mathbf{y}}

\def\hatg{\hat{\g}}

\begin{document}
\title{An optimal algorithm for stochastic strongly-convex optimization}
\author{ Elad Hazan\\
IBM Almaden\\
{\tt elad.hazan@gmail.com}
\and
Satyen Kale \\
Yahoo! Research \\
{\tt skale@yahoo-inc.com} }
\date{}
\maketitle

\begin{abstract}
We consider stochastic convex optimization with a strongly convex (but not necessarily smooth) objective. We give an algorithm which performs only gradient updates with optimal rate of convergence.
\end{abstract}

\section{Setup}

Consider the problem of minimizing a convex function on a convex domain $\K$:
$$ \min_{\x \in \K} f(\x). $$
Assume that we have an upper bound on the values of $f$, i.e. a number $M > 0$ such that for any $\bx_1, \bx_2 \in \K$, we have $|f(\bx_1) - f(\bx_2)| \leq M$. Also, assume we can compute an unbiased estimator of a subgradient of $f$ at any point $\x$, with $L_2$ norm bounded by some known value $G$. Assume that the domain $\K$ is endowed with a projection operator $\prod_{\K}(\y) = \arg \min_{\x \in \K} \|\x-\y\|$. Finally, we assume that $f$ satisfies the following inequality
$$ f(\x) - f(\x^*) \geq \lambda \|\x-\x^*\|^2 $$
where $\x^*$ is the point in $\K$ on which $f$ is minimized. This property holds, for example, if $f$ is $\lambda$-strongly-convex.
The canonical example of such an optimization problem is
support-vector-machine training.

\section{The algorithm}
The algorithm is a straightforward extension of stochastic gradient descent.
The new feature is the introduction of ``epochs" inside of which standard
stochastic gradient descent is used, but in each consecutive epoch the learning
rate decreases exponentially.

\begin{algorithm}[h!]
  \caption{\label{alg:online-variance-updates} Epoch-GD}
  \begin{algorithmic}[1]
    \STATE Input: parameters $M$,$G$, $\lambda$, $\eta_k$, $T_k$ and error tolerance  $\varepsilon$. Initialize $\x_1^1 \in \K $ arbitrarily.
    \FOR{$k  = 1$ to $\lceil \log_2 \frac{M}{\varepsilon}\rceil$}
    \STATE Set $V_k = \frac{M}{2^{k-1}}$, and start an epoch as follows
    \FOR{$t = 1$ to $T_k$} 
    \STATE Let the estimated subgradient of $f$ at $\x_t^k$ be $\hatg_t$ 
    \STATE Update
    $$
    \x_{t+1}^k = \prod_{\K}\left\{\x_t^k - {\eta_k} \hatg_t  \right\}
    $$ \label{alg:update}
    \ENDFOR
    \STATE Set $\x_1^{k+1} = \frac{1}{T_k} \sum_{t=1}^{T_k} \x_t^{k} $ (or pick one iterate at random).
    \ENDFOR
    \RETURN $\x_1^{k+1}$.
  \end{algorithmic}
\end{algorithm}

\section{Analysis}

Our main result is the following Theorem
\begin{Thm} \label{thm:main}
The final point $\x_1^{k+1}$ returned by the {\it Epoch-GD} algorithm, with
parameters $T_k = \left\lceil\frac{16G^2 }{\lambda V_k}\right\rceil$ and
$\eta_k = \frac{{V_k}}{4G^2}$,  has the property that $\E[f(\x_1^{k+1})] -
f(\x^\star) \leq \varepsilon$. The total number of gradient updates is
$O(\frac{G^2}{\lambda \varepsilon})$.
\end{Thm}

The inter-epoch use of standard gradient decent is analyzed using the following Lemma from \cite{Zinkevich03}:
\begin{Lem}[Zinkevich \cite{Zinkevich03}] \label{lem:zink}
Let $D = \|\x^* - \x_1\|_2$ and $\|\hatg_t\|\leq G$. Apply $T$ iterations of
the update $\x_{t+1} = \prod_{\K}\left\{\x_t - {\eta}  \hatg_t \right\}$. Then
$$\sum_{t=1}^T \hatg_t\cdot(\x_t - \x^*)  \leq  \eta  G^2 + \frac{D^2 }{\eta T}. $$
\end{Lem}

Now, if we set $\hatg_t$ to be the unbiased estimator of a subgradient $\g_t$
of $f$ at $\x_t$, then by the convexity of $f$, we get
$$f(\x_t) - f(\x^*)\ \leq\ \g_t\cdot (\x_t - \x^*) = \E_{t-1}[\hatg_t\cdot (\x_t - \x^*)],$$
where $\E_{t-1}[\cdot]$ denotes expectation conditioned on all the randomness
up to round $t-1$. This immediately implies the following:

\begin{Lem} \label{thm:batchzink}
Let $D = \|\x^* - \x_1\|_2$. Apply $T$ iterations of the update $\x_{t+1} =
\prod_{\K}\left\{\x_t - {\eta}  \hatg_t \right\} $, where $\hatg_t$ is an
unbiased estimator for the (sub)gradient of $f$ at $\x_t$ satisfying
$\|\hatg_t\|\leq G$. Then
$$ \frac{1}{T} \E[ \sum_{t=1}^T f({\x}_t)] - f(\x^*)  \leq  \eta  G^2 + \frac{D^2 }{\eta T } . $$
By convexity of $f$, we have the same bound for $ \E[ f(\bar{\x})] - f(\x^*)  $, where $\bar{\x} = \frac{1}{T} \sum_{t=1}^T \x_t$.
\end{Lem}
Define $\Delta_k = f({\x}_1^k) - f(\x^*)$. Using Theorem~\ref{thm:batchzink} we
prove the following key lemma:
\begin{Lem} \label{lem:boundOnQ}
For any $k$, we have $\E[\Delta_k] \leq V_k$.
\end{Lem}
\begin{proof}
We prove this by induction on $k$. The claim is true for $k = 1$ since $\Delta_k \leq M$. Assume that $\E[\Delta_k] \leq V_k$ for some $k \geq 1$ and now we prove it for $k+1$. For a random variable $X$ measurable w.r.t. the randomness defined up to epoch $k+1$, let $\E_k[X]$ denote its expectation conditioned on all the randomness up to phase $k$.
By Lemma \ref{thm:batchzink} we have
\begin{eqnarray*}
\E_k[ f(\x_1^{k+1})] - f(\x^*) & \leq  \eta_k  G^2 + \frac{ \|\x_1^k - \x^* \|^2}{\eta_k T_k}  \\
 & \leq  \eta_k  G^2 + \frac{ \Delta_k }{\eta_k  T_k  \lambda}  &\mbox{(by $\lambda$-strong convexity)}
\end{eqnarray*}
and hence,
$$\E[\Delta_{k+1}]\ \leq\  \eta_k G^2 + \frac{ \E[\Delta_k] }{\eta_k T_k \lambda} \ \leq\ \eta_k  G^2 + \frac{ V_k }{\eta_k T_k \lambda} \ \leq\ \frac{V_k}{2}\ =\ V_{k+1},$$
as required. The second inequality uses the induction hypothesis, and the last
inequality and equality use the definition of $V_k$ and the values  $\eta_k = \frac{{V_k}}{4 G^2} $ and $T_k = \lceil\frac{ 16 G^2 }{\lambda V_k}\rceil$.

\end{proof}

We can now prove our main theorem:
\begin{proof}[Proof of Theorem~\ref{thm:main}.]
By the previous claim, taking $k = \lceil \log_2 \frac{M}{\varepsilon}\rceil$
we have
$$ \E[ f(\x^{k+1}_1)] - f(\x^*)\ =\ \E[ \Delta_{k+1}]\ \leq\ V_{k+1}\ =\ \frac{M}{2^k}\ \leq\ \varepsilon,$$
as claimed.

To compute the total number of gradient updates, we sum up along the epochs: in
each epoch $k$ we have $T_k=\lceil \frac{ 16 G^2   }{\lambda V_k}\rceil$ gradient updates,
for a total of
$$ \sum_{k=1}^{\lceil \log_2 \frac{M}{\varepsilon} \rceil}  \left\lceil\frac{16G^2}{\lambda V_k}\right\rceil\ \leq\ \sum_{k=1}^{\lceil \log_2\frac{M}{\varepsilon} \rceil}  \frac{16G^2 \cdot 2^{k-1}}{\lambda M} + 1\ \leq\ \frac{20G^2}{\lambda
\epsilon},$$
assuming that $\lceil \log_2 \frac{M}{\varepsilon} \rceil \leq
\frac{2G^2}{\lambda \varepsilon}$.
\end{proof}

\section{Conclusions}

Extension of the above result to stochastic optimization of strongly convex functions with respect to norms other than the Euclidean norm are straightforward via standard online learning techniques. A factor two speedup can be obtained by stoping the epoch at a random point.

We thank Nati Srebro for bringing the problem of deriving an efficient attention algorithm for stochastic strongly-convex optimization to our attention.

\bibliographystyle{plain}
\bibliography{svm}
\appendix

\section{High probability bounds}

We briefly sketch how using essentially the same algorithm with slightly more
iterations, we can get a high probability guarantee on the quality of the
solution. The update in line 6 requires a projection onto a smaller set, and
becomes
$$ \text{6: Update }\ \x_{t+1}^k = \prod_{\K \cap B_{V_k} (\x_1^k)}\left\{\x_t^k - {\eta_k} \hatg_t  \right\} $$
Here $B_r(\x)$ denotes the $L_2$ ball of radius $r$ around the point $x$. We
assume that such a projection can be computed very efficiently. In particular,
if $\K = \mathbb{R}^n$, then the projection is simply a scaling down of the
vector towards the center of the ball.

We prove:
\begin{Thm} \label{thm:main2}
The final point $\x_1^{k+1}$ returned by the modified {\it Epoch-GD} algorithm,
with parameters $T_k = \left\lceil\frac{100 G^2 \ln(1/\deltil)}{\lambda
V_k}\right\rceil$, $\eta_k = \frac{V_k}{10 G^2}$, where $\deltil = \frac{\delta }{4 \log \frac{M}{\varepsilon}}$, has the property that $f(\x_1^{k+1})
- f(\x^\star) \leq \varepsilon$ with probability at least $1-\delta$. The total
number of gradient updates is $O(\frac{G^2\log(1/\deltil)}{\lambda \varepsilon})$.
\end{Thm}

The following Lemma is analogous to Lemma \ref{thm:batchzink}, but provides a
high probability guarantee.
\begin{Lem} \label{lem:highprobGD}
Let $D = \|\x^* - \x_1\|_2$. Apply $T$ iterations of the update $\x_{t+1} =
\prod_{\K \cap B_{D}(\x_1) }\left\{\x_t - {\eta} \hatg_t \right\} $, where
$\hatg_t$ is an unbiased estimator for the (sub)gradient of $f$ at $\x_t$
satisfying $\|\hatg_t\|\leq G$. Then with probability at least $1-\deltil$
$$\frac{1}{T}\sum_{t=1}^T f(\x_t) - f(\x^*)\  \leq\ \eta G^2 +  \frac{D^2 }{\eta T }  + \frac{8GD \sqrt{\ln(1/\deltil)}}{\sqrt{T}}.$$
By the convexity of $f$, the same bound holds for $f(\bar{\x}) - f(\x^*)$,
where $\bar{\x} = \frac{1}{T} \sum_{t=1}^T \x_t$.
\end{Lem}
\begin{proof}
Let $\g_t = \E_{t-1}[\hatg_t]$, a subgradient of $f$ at $\x_t$, where
$\E_{t-1}[\cdot]$ denotes the expectation conditioned on all randomness up to
round $t-1$. Consider the martingale difference sequence given by
$$X_t\ =\ \g_t\cdot (\x_t - \x^*) - \hatg_t \cdot (\x_t - \x^*).$$
We can bound $|X_t|$ as follows:
$$|X_t| \leq \| \hatg_t \| \| \x_t - \x^*\| + \E_{t-1}[\| \hatg_t \|] \| \x_t - \x^*\|\ \leq\ 4 G D,$$
where the last inequality uses the fact that $\x_t \in B_{D}(\x_1)$, and
hence by the triangle inequality $\|\x_t - \x^*\| \leq \| \x_t - \x_1 \| +
\|\x_1 - \x^*\| \leq 2D$.

By Azuma's inequality (see Lemma~\ref{lem:azuma}), with probability at least
$1-\deltil$, the following holds:
\begin{equation} \label{eq:azuma}
\frac{1}{T}\sum_{t=1}^T  \g_t\cdot (\x_t -
\x^*) - \frac{1}{T}  \sum_{t=1}^T \hatg_t\cdot (\x_t - \x^*)\  \leq\  \frac{8GD
\sqrt{\ln(1/\deltil)}}{\sqrt{T}}.
\end{equation}
 Note that by the convexity of $f$, we have
$f(\x_t) - f(\x^*) \leq \g_t\cdot(\x_t - \x^*)$. Then, by using
Lemma~\ref{lem:zink} and inequality~(\ref{eq:azuma}), we get the claimed
bound.
\end{proof}

We can now proceed along the same lines as Theorem \ref{thm:main} and prove the same result with high probability, the derivation is completely analoguous.
\begin{Lem} \label{lem:hiprobboundOnQ}
For an appropriate choice of $\eta_k, T_k$, the following holds. For any $k$,
with probability $(1-\deltil)^k$ we have $\Delta_k \leq V_k$.
\end{Lem}
\begin{proof}
We prove this by induction on $k$. The claim is true for $k = 1$ since
$\Delta_k \leq M$. Assume that $\Delta_k \leq V_k$ for some $k \geq 1$ with
probability at least $(1 - \deltil)^k$ and now we prove it for $k+1$. We
condition on the event that $\Delta_k \leq V_k$. By Lemma \ref{lem:highprobGD},
we have with probability at least $1-\deltil$,
\begin{align*}
\Delta_{k+1} & = f(\x_1^{k+1}) - f(\x^*) \\
& \leq  \eta_k  G^2 + \frac{ \|\x_1^k - \x^* \|^2}{\eta_k T_k} + \frac{ 8G \|\x_1^k - \x^*\| \sqrt{\ln(1/\deltil)}}{\sqrt{T_k}} & \mbox{(by Lemma \ref{lem:highprobGD})}\\
 & \leq  \eta_k  G^2 + \frac{ \Delta_k }{\eta_k  T_k  \lambda} + \frac{ 8G \sqrt{\Delta_k} \sqrt{\ln(1/\deltil)}}{\sqrt{\lambda T_k}} &\mbox{(by $\lambda$-strong convexity)} \\
 & \leq  \eta_k  G^2 + \frac{ V_k }{\eta_k  T_k  \lambda} + \frac{ 8G \sqrt{V_k} \sqrt{\ln(1/\deltil)}}{\sqrt{\lambda T_k}} &\mbox{(by the conditioning)}
 \end{align*}
Let $T_k = \left\lceil\frac{100 G^2 \ln(1/\deltil) }{\lambda V_k}\right\rceil$,
and we get
$$f(\x_1^{k+1}) - f(\x^*)\ \leq \ \eta_k  G^2 +  \frac{1}{\eta_k} \frac{ V_k^2
}{100 G^2  \ln(1/\deltil)} +  \frac {V_k}{10}$$ Next set $\eta_k = \frac{V_k}{10
G^2}$, and we get that
$$\Delta_{k+1} = f(\x_1^{k+1}) - f(\x^*)\ \leq\ \frac{V_k}{10} + \frac{ V_k }{10 \ln(1/\deltil)} +  \frac {V_k}{10} \leq \frac{V_k}{2} =
V_{k+1}.$$ Factoring in the conditioned event, which happens with probability
at least $(1-\deltil)^k$, overall, we get that $\Delta_{k+1} \leq V_{k+1}$ with
probability at least $(1-\deltil)^{k+1}$.
\end{proof}

We can now prove our high probability theorem:
\begin{proof}[Proof of Theorem \ref{thm:main2}]
By the previous claim, taking $k = \lceil \log_2 \frac{M}{\varepsilon}\rceil$
we have with probability at least $(1-\tilde{\delta})^k$ that
$$  f(\x^{k+1}_1) - f(\x^*)\ =\  \Delta_{k+1}\ \leq\ V_{k+1}\ =\ \frac{M}{2^k}\ \leq\ \varepsilon,$$
Since $\tilde{\delta} = \frac{\delta}{4k}$, and hence $(1-\tilde{\delta})^k \geq 1 - \delta$ as needed.

To compute the total number of gradient updates, we sum up along the epochs: in
each epoch $k$ we have $T_k=O(\frac{ G^2 \log(1/\tilde{\delta}) }{\lambda V_k})
$ gradient updates, for a total of
\begin{align*}
\sum_{k=1}^{\lceil \log_2 \frac{M}{\varepsilon} \rceil}  O\left(\frac{ G^2 \log(1/\deltil)}{\lambda V_k}\right)
&= \sum_{k=1}^{\lceil \log_2\frac{M}{\varepsilon} \rceil}  O\left(\frac{ 2^{k-1}G^2 \log(1/\deltil)}{\lambda M}\right)\\
&= O\left( \frac{ G^2 \log({1}/{\deltil})}{\lambda \varepsilon}\right).
\end{align*}
\end{proof}

\subsection{Martingale concentration lemma}
The following inequality is standard in obtaining high probability regret bounds:
\begin{Lem}[Azuma's inequality] \label{lem:azuma}
Let $X_1, \ldots, X_T$ be a martingale difference sequence. Suppose that $|X_t|
\leq b$. Then, for $\delta > 0$, we have
$$ \Pr\left[\sum_{t=1}^T X_t \geq \sqrt{2 b^2 T \ln(1/\delta)}   \right]\ \leq\ \delta.$$
\end{Lem}

\end{document}